\documentclass[smallcondensed,numbook]{svjour3}
\usepackage[table]{xcolor}
\usepackage{color}
\usepackage{amsmath}
\usepackage{amsfonts}
\usepackage{amssymb}
\usepackage{mathrsfs}
\usepackage{tabularx}
\usepackage{graphicx, subfigure}
\usepackage{enumitem}
\usepackage{pgfplots}
\pgfplotsset{width=7cm}

\usepackage{tikz, ifthen}
\usetikzlibrary{shapes.misc}

\newcommand{\undertilde}[1]{\ensuremath{\mathord{\vtop{\ialign{##\crcr
   $\hfil\displaystyle{#1}\hfil$\crcr\tikzstyle{largegraph}=[
  every node/.style={circle,draw,fill=black,inner sep=0pt,minimum width=6pt},
  every path/.append style={semithick}
]\noalign{\kern1.5pt\nointerlineskip}
   $\hfil\tilde{}\hfil$\crcr\noalign{\kern1.5pt}}}}}}

\tikzstyle{smallgraph}=[
  every node/.style={circle,draw,fill=black,inner sep=0pt,minimum width=0pt},
  every path/.append style={semithick}
]

\tikzstyle{largegraph}=[
  every node/.style={circle,draw,fill=black,inner sep=0pt,minimum width=2pt},
  every path/.append style={semithick}
]
\pgfmathsetmacro{\threevradius}{0.375*sqrt(2/3)}

\pgfmathsetmacro{\threevradius}{0.375*sqrt(2/3)}

\newcommand\inS[2]{
\tikzset{inS/.style={circle,draw,fill=black,inner sep=0pt,minimum width=8pt}}
\ifthenelse{\equal{#1}{a}}{\node[inS] at (#2,1) {};}{};
\ifthenelse{\equal{#1}{b}}{\node[inS] at (#2,0.5) {};}{};
\ifthenelse{\equal{#1}{c}}{\node[inS] at (#2,1.5) {};}{};
\ifthenelse{\equal{#1}{d}}{\node[inS] at (#2,2) {};}{};
}

\newcommand\inSa[2]{
\tikzset{inS/.style={circle,draw,fill=white,inner sep=0pt,minimum width=8pt}}
\ifthenelse{\equal{#1}{a}}{\node[inS] at (#2,1) {};}{};
\ifthenelse{\equal{#1}{b}}{\node[inS] at (#2,0.5) {};}{};
\ifthenelse{\equal{#1}{c}}{\node[inS] at (#2,1.5) {};}{};
\ifthenelse{\equal{#1}{d}}{\node[inS] at (#2,2) {};}{};
}

\newcommand\notS[2]{
\tikzset{inS/.style={cross out,draw,minimum size=8*(\pgflinewidth), inner sep=0pt, outer sep=0pt}}
\ifthenelse{\equal{#1}{a}}{\node[inS] at (#2,1) {};}{};
\ifthenelse{\equal{#1}{b}}{\node[inS] at (#2,0.5) {};}{};
\ifthenelse{\equal{#1}{c}}{\node[inS] at (#2,1.5) {};}{};
\ifthenelse{\equal{#1}{d}}{\node[inS] at (#2,2) {};}{};
}

\newcommand\notSa[2]{
\tikzset{inS/.style={strike out,draw,minimum size=8*(\pgflinewidth), inner sep=0pt, outer sep=0pt}}
\ifthenelse{\equal{#1}{a}}{\node[inS] at (#2,1) {};}{};
\ifthenelse{\equal{#1}{b}}{\node[inS] at (#2,0.5) {};}{};
\ifthenelse{\equal{#1}{c}}{\node[inS] at (#2,1.5) {};}{};
\ifthenelse{\equal{#1}{d}}{\node[inS] at (#2,2) {};}{};
}

\renewenvironment{proof}{\par {\sc {\bf Proof.}\hskip 5pt}}{\hfill \qed \par}

\begin{document}

\title{Binary Programming Formulations for the Upper Domination Problem}
\author{Ryan Burdett, Michael Haythorpe, Alex Newcombe}

\institute{Ryan Burdett
\at Flinders University, 1284 South Road, Tonsley Park, SA, Australia\\
\email{ryan.burdett@flinders.edu.au} \and Michael Haythorpe (Corresponding author)
\at Flinders University, 1284 South Road, Tonsley Park, SA, Australia, Ph: +61 8 8201 2375, Fax: +61 8 8201 2904\\
\email{michael.haythorpe@flinders.edu.au} \and Alex Newcombe
\at Flinders University, 1284 South Road, Tonsley Park, SA, Australia\\
\email{alex.newcombe@flinders.edu.au}} \maketitle

\begin{abstract}We consider Upper Domination, the problem of finding the minimal dominating set of maximum cardinality. Very few exact algorithms have been described for solving Upper Domination. In particular, no binary programming formulations for Upper Domination have been described in literature, although such formulations have proved quite successful for other kinds of domination problems. We introduce two such binary programming formulations, and show that both can be improved with the addition of extra constraints which reduce the number of feasible solutions. We compare the performance of the formulations on various kinds of graphs, and demonstrate that (a) the additional constraints improve the performance of both formulations, and (b) the first formulation outperforms the second in most cases, although the second performs better for very sparse graphs. Also included is a short proof that the upper domination number of any generalized Petersen graph $P(n,k)$ is equal to $n$.

\end{abstract}

\section{Introduction}

Consider an undirected graph $G = (V,E)$. A {\em dominating set} for $G$ is a set of vertices $S \subseteq V$ such that for every vertex $v \in V \setminus S$, there exists an adjacent vertex $w \in S$. The problem of finding a dominating set of minimum cardinality for a given graph is known as the {\em domination problem}, and the size of the minimum dominating set is the {\em domination number} of the graph, typically denoted $\gamma(G)$. The domination problem is one of the most widely studied problems in graph theory and computational complexity, and numerous algorithms have been described to find minimum dominating sets.

A dominating set $S$ is said to be {\em minimal} if no proper subset of $S$ is a dominating set. Clearly, any minimum dominating set is minimal, however some larger dominating sets may also be minimal. Then, the {\em upper domination problem} is the problem of finding the minimal dominating set of maximum cardinality for a given graph $G$. The cardinality of the latter is denoted by $\Gamma(G)$. The upper domination problem can be seen as a maximin version of the domination problem, and is hence useful in the analysis of the worst-case performance of algorithms for the latter. The upper domination problem is known to be NP-hard in general, and in particular is NP-hard for planar cubic graphs \cite{bazgan}.

There are a few algorithms designed to solve the upper domination problem for certain restricted classes of graphs. It can be easily solved for graphs with maximum degree 2. Bazgan et al. \cite{bazgan} give an exact $O^*(1.348^n)$ algorithm for subcubic graphs, as well as an exact $O(7^p)$ algorithm for graphs of pathwidth $p$. For general graphs, the best known algorithm is an enumeration algorithm due to Fomin et al. \cite{fomin} which computes {\em every} minimal dominating set in $O^*(1.7159^n)$ time, and in doing so gives the same bound on the number of minimal dominating sets in a graph of order $n$.

One common approach to constructing an exact algorithm for an NP-hard problem is to model it using mixed-integer linear programming (MILP). There are numerous highly optimised MILP solvers, with CPLEX being perhaps the most notable of these. For the standard domination problem, such a model can be easily constructed. First, we define $N(v)$ to be the set of vertices adjacent to $v$, and $N[v] := N(v) \cup \{v\}$. Then we introduce binary variables $x_v \in \{0,1\}$, $\forall v \in V$, with the interpretation that $v \in S$ if and only if $x_v = 1$. It is then easy to see that the domination problem is equivalent to the following:

$$\min \sum_{v \in V} x_v$$
$$\mbox{subject to}$$
\begin{eqnarray*}\sum_{w \in N[v]} x_w & \geq & 1, \;\;\;\;\;\;\;\;\; \forall v \in V.
\end{eqnarray*}

MILP formulations have been described in the literature for several variants of domination, such as total domination \cite{burger}, Roman domination \cite{burger,ivanovic,revelle}, weak Roman domination \cite{burger,ivanovic2}, secure domination \cite{burdett,burger}, connected domination \cite{fan,simonetti}, and power domination \cite{fan}. In particular, Burger et al. \cite{burger} noted that even an unsophisticated approach to solving the MILP formulations for secure domination (i.e. simply submitting it to CPLEX) significantly outperformed the previous best-known exact algorithms from \cite{burger2}, and indicated that better formulations or more sophisticated approaches such as column generation techniques could further improve the results; indeed, a superior formulation which can be solved even faster was given in \cite{burdett}.

In this paper, we add to this literature by describing, for the first time, MILP formulations for the upper domination problem. In particular, we introduce two formulations. The first has $2|V|$ binary variables and $3|V|$ constraints, the second has $2|V| + |E|$ binary variables and $3|V| + |E|$ constraints. We then proceed to show that both formulations can be improved with the addition of some extra constraints; there are $|V|$ new constraints for the first formulation, and $|V| + |E|$ new constraints for the second formulation. We compare these formulations experimentally, and observe that (a) the additional constraints lead to improved performance, and (b) the first formulation generally outperforms the second for general graphs; however, the second formulation does perform better for very sparse graphs.

\section{Formulations\label{sec-formulations}}

We now introduce two MILP formulations for the upper domination problem. In particular, these are binary programming formulations, as all variables in both formulations will be binary.

\subsection{Formulation 1}

In the following formulation, we will use two vectors of binary variables, ${\bf x}$ and ${\bf z}$ to find a minimal dominating set $S$ of maximum cardinality. The intended interpretation will be that $x_v = 1$ if and only if $v$ is contained in $S$, and $z_v = 1$ if and only if $|S \cap N[v]| \geq 2$. In situations where $\{w\} = S \cap N[v]$, it is common to refer to $v$ as a {\em private neighbour} of $w$. Hence, the intention is that $z_v = 0$ if and only if there is a vertex $w \in N[v]$ such that $v$ is a private neighbour of $w$. We will use $d(v)$ to denote the degree of vertex $v$.

$$\max \sum_{v \in V} x_v$$
$$\mbox{subject to}$$
\begin{eqnarray}\sum_{w \in N[v]} x_w & \geq & 1, \;\;\;\qquad\qquad \forall v \in V,\label{f1-1}\\
\sum_{w \in N[v]} x_w - d(v)z_v & \leq & 1, \;\;\;\qquad\qquad \forall v \in V,\label{f1-2}\\
x_v + \sum_{w \in N[v]} z_w & \leq & d(v) + 1, \qquad \forall v \in V,\label{f1-3}
\end{eqnarray}


\begin{theorem}Constraints (\ref{f1-1})--(\ref{f1-3}) provide a correct formulation for the upper domination problem.\label{thm-f1}\end{theorem}

\begin{proof}We first argue that any solution $({\bf x}, {\bf z})$ to the constraints (\ref{f1-1})--(\ref{f1-3}) corresponds to a minimal dominating set. Since ${\bf x}$ is binary, we can define $S({\bf x})$ to be the set of all vertices $v$ for which $x_v = 1$, and it is clear that from (\ref{f1-1}) that $S({\bf x})$ is a dominating set. Then, recall that a dominating set $S$ is minimal if no proper subset of $S$ is a dominating set. Hence it is sufficient to confirm that, for every vertex $v \in S({\bf x})$, the set $S({\bf x}) \setminus \{v\}$ is not dominating.

Consider any vertex $v \in S({\bf x})$, then by definition we have $x_v = 1$. From (\ref{f1-3}) this implies that $\sum_{w \in N[v]} z_w \leq d(v)$, which means there is at least one vertex, say $u \in N[v]$, such that $z_u = 0$. But from constraints (\ref{f1-2}), this implies that $\sum_{w \in N[u]} x_w \leq 1$. Since $x_v = 1$, this simplifies to $\sum_{w \in N[u] \setminus \{v\}} x_w \leq 0$. Hence, $S({\bf x}) \cap N[u] = \{v\}$, and so $S({\bf x}) \setminus \{v\}$ is not dominating. Since this argument can be applied to every vertex in $S({\bf x})$, we conclude that $S({\bf x})$ is a minimal dominating set.

Next, we argue that for every minimal dominating set, there is a corresponding solution to the constraints (\ref{f1-1})--(\ref{f1-3}). Suppose that $S$ is a minimal dominating set, then we need to assign values to ${\bf x}$ and ${\bf z}$ that will satisfy (\ref{f1-1})--(\ref{f1-3}). For all $v \in V$, we assign $x_v = 1$ if $v \in S$, and $x_v = 0$ if $v \not\in S$. Furthermore, for every vertex $v \in S$, there must be at least one vertex, say $w \in N[v]$ such that $S \cap N[w] = \{v\}$. For each $v \in S$ we select all such vertices $w$ and assign $z_w = 0$. Then, once we have finished considering all $v \in S$, any unassigned $z_w$ values are set to 1. Note that it is clear that (\ref{f1-1}) is satisfied, since it only contains $x$ variables and $S$ is a dominating set. We will now argue that constraints (\ref{f1-2}) and (\ref{f1-3}) are also satisfied.

Consider constraints (\ref{f1-2}). For any given $v \in V$, the constraint could only be violated if $z_v = 0$. However, according to the assignment of values above, this implies that $|S \cap N[v]| = 1$. Hence $\sum_{w \in N[v]} x_w = 1$ and so the constraint is not violated for any valid choice of minimal dominating set $S$.

Finally, consider constraints (\ref{f1-3}). For any given $v \in V$, this constraint could only be violated if $x_v = 1$ and $\sum_{w \in N[v]} z_w = d(v) + 1$. However, according to the assignment of values above, this means there is no vertex $w \in N[v]$ such that $S \cap N[w] = \{v\}$. Hence, $S \setminus \{v\}$ is still a dominating set, which contradicts the assumption that $S$ is minimal. Hence, this constraint is not violated for any valid choice of minimal dominating set $S$.\end{proof}


\subsection{Formulation 2}

In the following formulation, in addition to the $x_v$ variables from Formulation 1, we introduce a vector {\bf y}, containing binary variables $y_{vw}$ defined for all $v \in V$, and $w \in N[v]$. The intended interpretation will be that if $y_{vw} = 1$, then vertex $w$ is dominated {\em only} by vertex $v$; that is, $w$ is a private neighbour of $v$. We again use $d(v)$ to denote the degree of vertex $v$.

$$\max \sum_{v \in V} x_v$$
$$\mbox{subject to}$$
\begin{eqnarray}\sum_{w \in N[v]} x_w & \geq & 1, \;\;\;\;\;\qquad \forall v \in V,\label{f2-1}\\
x_v - \sum_{w \in N[v]} y_{vw} & \leq & 0, \;\;\;\;\;\qquad \forall v \in V,\label{f2-3}\\
d(w) y_{vw} + \sum_{u \in N[w] \setminus v} x_u & \leq & d(w), \qquad \forall v \in V, w \in N[v].\label{f2-4}\end{eqnarray}

\begin{theorem}Constraints (\ref{f2-1})--(\ref{f2-4}) provide a correct formulation for the upper domination problem.\label{thm-f2}\end{theorem}

\begin{proof}We first argue that any solution $({\bf x}, {\bf y})$ to the constraints (\ref{f2-1})--(\ref{f2-4}) corresponds to a minimal dominating set. Since ${\bf x}$ is binary, we define $S({\bf x})$ to be the set of all vertices $v$ for which $x_v = 1$. It is clear that from (\ref{f2-1}) that $S({\bf x})$ is a dominating set. Then, recall that a dominating set $S$ is minimal if no proper subset of $S$ is a dominating set. Hence it is sufficient to confirm that, for every vertex $v \in S({\bf x})$, the set $S({\bf x}) \setminus \{v\}$ is not dominating.

Consider any vertex $v \in S({\bf x})$, then by definition we have $x_v = 1$. From (\ref{f2-3}) this implies that $\sum_{w \in N[v]} y_{vw} \geq 1$. Hence, there is some $w \in N[v]$ such that $y_{vw} = 1$. Then, consider (\ref{f2-4}), for the corresponding values of $v$ and $w$. Since $y_{vw} = 1$, (\ref{f2-4}) simplifies to $\sum_{u \in N[w] \setminus v} x_u \leq 0$. Hence, $S({\bf x}) \cap N[w] = \{v\}$, and so $S({\bf x}) \setminus \{v\}$ is not dominating. Since this argument can be applied to every vertex in $S({\bf x})$, we conclude that $S({\bf x})$ is a minimal dominating set.

Next, we argue that for every minimal dominating set, there is a corresponding solution to the constraints (\ref{f2-1})--(\ref{f2-4}). Suppose that $S$ is a minimal dominating set, then we need to assign values to ${\bf x}$ and ${\bf y}$ that will satisfy (\ref{f2-1})--(\ref{f2-4}). For all $v \in V$, we assign $x_v = 1$ if $v \in S$, and $x_v = 0$ if $v \not\in S$. Furthermore, for every vertex $v \in S$, there must be at least one vertex, say $w \in N[v]$ such that $w$ is a private neighbour of $v$. For each $v \in S$ we can select exactly one such vertex $w$, and denote it by $p(v)$. Then, for all $v \in V$, $w \in N[v]$, we assign $y_{vw} = 1$ if $v \in S$ and $w = p(v)$, and $y_{vw} = 0$ otherwise. Note that it is clear that (\ref{f2-1}) is satisfied, since it only contains $x$ variables and $S$ is a dominating set. We will now argue that constraints (\ref{f2-3}) and (\ref{f2-4}) are also satisfied.


Consider constraints (\ref{f2-3}). For any given $v$ and $w$, the constraint could only be violated if $x_v = 1$ and $\sum_{w \in N[v]} y_{vw} = 0$. However, if $x_v = 1$, then $v \in S$, and hence in the assignment of values above, we have chosen a private neighbour of $v$, say $p(v) \in N[v]$, and assigned $y_{vp(v)} = 1$. Hence, $\sum_{w \in N[v]} y_{vw} \geq 1$ and so this constraint is not violated for any valid choice of minimal dominating set $S$.

Finally, consider constraints (\ref{f2-4}). It is clear that $\sum_{u \in N[w] \setminus v} x_u \leq d(w)$. Hence, for any given $v$ and $w$, the constraint could only be violated if $y_{vw} = 1$ and $\sum_{u \in N[w] \setminus v} x_u \geq 1$. However, the latter implies there is some vertex $u \in N[w] \setminus v$ such that $u \in S$. That is, $w$ is not a private neighbour of $v$, and so $y_{vw} = 0$. Hence, this constraint is not violated for any valid choice of minimal dominating set $S$.\end{proof}

\section{Additional Constraints\label{sec-additional}}

In the previous section, we introduced Formulations 1 and 2 and showed that they are equivalent to the upper domination problem. In this section, we show that these formulations can each be augmented with an additional set of constraints which tighten the formulations, and hence may lead to superior empirical results; we explore this further in Section \ref{sec-results}.

\subsection{Augmented Formulation 1}

Recall that in addition to the standard ${\bf x}$ variables corresponding to vertices, Formulation 1 included an extra set of variables, {\bf z}. The intended interpretation of ${\bf z}$ is that $z_v = 1$ if and only if $|S \cap N[v]| \geq 2$, however the constraints in Formulation 1 only imply this relationship in one direction. Specifically, they imply that if $|S \cap N[v]| \geq 2$ then $z_v = 1$. This can be seen by considering constraint (\ref{f1-2}). If we have $|S \cap N[v]| \geq 2$ then that means we have $\sum_{w \in N[v]} x_w \geq 2$. Constraints (\ref{f1-2}) then reduce to $d(v)z_v \geq 1$. Since $z_v$ is binary, we must have $z_v = 1$.

In order to imply the opposite direction, we can impose additional constraints. In the following, it is clear that if $z_v = 1$, then the constraint reduces to $\sum_{w \in N[v]} x_w \geq 2$, as desired.

\begin{eqnarray}\sum_{w \in N[v]} x_w - z_v & \geq & 1, \;\qquad\qquad \forall v \in V.\label{f1-4}\end{eqnarray}

We will refer to Formulation 1 with the addition of constraints (\ref{f1-4}) as {\em Augmented Formulation 1}.

\begin{theorem}Augmented Formulation 1 is equivalent to the upper domination problem.\end{theorem}

\begin{proof}Since Augmented Formulation 1 contains all the constraints from Formulation 1, then from Theorem \ref{thm-f1} we know that any solution must correspond to a minimal dominating set. Then all that remains is to ensure that for each minimal dominating set, there is a corresponding solution to Formulation 1 which also satisfies constraints (\ref{f1-4}).

Consider any minimal dominating set in the graph, and assign values to ${\bf x}$ and ${\bf z}$ in the manner described in the proof of Theorem \ref{thm-f1}. As proved in Theorem \ref{thm-f1} this is a solution of Formulation 1. Since this solution satisfies constraints (\ref{f1-1}), we know that $\sum_{w \in N[v]} x_w \geq 1$. Then, since the solution is binary, for any given $v \in V$, constraints (\ref{f1-4}) could only be violated if $z_v = 1$ and $\sum_{w \in N[v]} x_w = 1$. However, the latter implies that $|S \cap N[v]| = 1$, and according to the assignment of values in the proof of Theorem \ref{thm-f1}, we would have set $z_v = 0$. Hence, this constraint is not violated, completing the proof.\end{proof}

\subsection{Augmented Formulation 2}

We now undergo a similar exercise for Formulation 2, which contains the extra set of variables ${\bf y}$. The intended interpretation of ${\bf y}$ is that $y_{vw} = 1$ if and only if $w$ is a private neighbour of $v$. However, the constraints of Formulation 2 only imply this relationship in one direction. Specifically, they imply that if $y_{vw} = 1$ then $w$ is a private neighbour of $v$. This can be seen by considering constraint (\ref{f2-4}). If we have $y_{vw} = 1$ then the constraint reduces to $\sum_{u \in N[w] \setminus v} x_u = 0$. This, combined with constraints (\ref{f2-1}) and the variables being binary implies that $x_v = 1$ and hence $w$ is a private neighbour of $v$.

In order to imply the opposite direction, we can impose additional constraints. In the following, it is clear that if $\sum_{u \in N[w] \setminus v} x_u = 0$ then the constraint reduces to $y_{vw} \geq 1$, and since $y_{vw}$ is binary it implies $y_{vw} = 1$.

\begin{eqnarray}y_{vw} + \sum_{u \in N[w] \setminus v} x_u & \geq & 1, \;\qquad\qquad \forall v \in V, w \in N[v].\label{f2-5}\end{eqnarray}

We will refer to Formulation 2 with the addition of constraints (\ref{f2-5}) as {\em Augmented Formulation 2}.

\begin{theorem}Augmented Formulation 2 is equivalent to the upper domination problem.\end{theorem}

\begin{proof}Since Augmented Formulation 2 contains all the constraints from Formulation 2, then from Theorem \ref{thm-f2} we know that any solution must correspond to a minimal dominating set. Then all that remains is to ensure that for each minimal dominating set, there is a corresponding solution to Formulation 2 which also satisfies constraints (\ref{f2-5}).

Consider any minimal dominating set in the graph, and assign values to ${\bf x}$ and ${\bf y}$ in the manner described in the proof of Theorem \ref{thm-f2}. As proved in Theorem \ref{thm-f2} this is a solution of Formulation 2. Since this solution is binary, for any given $v \in V$, $w \in N[v]$, constraints (\ref{f2-5}) could only be violated if $y_{vw} = 0$ and $\sum_{u \in N[w] \setminus v} x_u = 0$. However, the latter along with constraints (\ref{f2-1}) imply that $x_v = 1$ and hence $w$ is a private neighbour of $v$. Then, according to the assignment of values in the proof of Theorem \ref{thm-f2}, we would have set $y_{vw} = 1$. Hence, this constraint is not violated, completing the proof.\end{proof}

We now provide a simple example to demonstrate that the augmented formulations are tighter than the original formulations.

\begin{example}Consider the complete graph $K_4$ and a minimal dominating set $S = \{1\}$. The intended solution corresponding to this minimal dominating set for Formulation 1 is $x_1 = 1$, and all other variables are zero. Likewise, the intended solution corresponding to this minimal dominating set for Formulation 2 is $x_1 = y_{11} = y_{12} = y_{13} = y_{14} = 1$, and all other variables are 0. However, it can be easily checked that Formulation 1 is also satisfied by setting $x_1 = 1$, and letting up to three of $z_1$, $z_2$, $z_3$ and $z_4$ be equal to 1 and the others to zero. Likewise, Formulation 2 is also satisfied by setting $x_1 = 1$, and ensuring at least one of $y_{11}$, $y_{12}$, $y_{13}$ and $y_{14}$ is equal to 1 while the others are zero. Hence for both formulations, there are multiple solutions corresponding to this one minimal dominating set. However, it can be checked that only the intended solutions satisfy the additional constraints in the augmented formulations.\end{example}

\section{Experimental Results\label{sec-results}}

We now present some experimental results comparing the performance of the two formulations and their augmented versions. Each instance considered was run using CPLEX v12.9.0 via the Concert environment to MATLAB R2022b update 1 (9.13.0.2080170) 64-bit. The experiments were conducted on an Intel(R) Core(TM) i5-12500 CPU with a 6 core, 3.00GHz processor and 16GB of RAM, running Windows 10 Enterprise version 22H2. In each case, we set a maximum time limit of 10,000 seconds. It is worth noting that although there is an overhead involved with using the Concert environment in MATLAB, the times we report here are solely CPLEX runtimes.

There are a number of graphs for which the upper domination number is trivial. For instance, for complete graphs $K_n$, it is easy to see that $\Gamma(K_n) = 1$. Likewise, for complete bipartite graphs $K_{mn}$ with $2 \leq m \leq n$ it is also easy to see that $\Gamma(K_{mn}) = 2$. However, these are not very interesting (or challenging) instances. Instead, we will focus on the following graph families.

\begin{enumerate} \item The $2 \times k$ queen graph, $Q_{2k}$. From \cite{hedetniemi} we have $\Gamma(Q_{2k}) = \left\lceil\frac{k}{2}\right\rceil$ for $k \geq 1$.
\item The $2 \times k$ rook graph, $K_2 \Box K_k$. From \cite{burcroff} we have $\Gamma(K_2 \Box K_k) = k$ for $k \geq 1$.
\item The $k \times k$ rook graph, $K_k \Box K_k$. From \cite{fricke} we have that $\Gamma(K_k \Box K_k) = k$, for $k \geq 1$.
\item The $k \times k$ bishop graph, $B_{kk}$. From \cite{fricke} we have $\Gamma(B_{kk}) = 2k - 2$, for $k \geq 2$.
\item The $k \times k$ knight graph, $C_{kk}$. From \cite{cockayne} we have $\Gamma(C_{kk}) = \left\lceil\frac{k^2}{2}\right\rceil$ for $k \geq 3$.
\item The $k \times k$ grid graph, $G_{kk}$. The upper domination number is not known in general for grid graphs.
\item The flower snark, $J_k$. From \cite{flowerdom} we have $\Gamma(J_k) = 2k$ for even $k \geq 4$, and $\Gamma(J_k) = 2k - 1$ for odd $k \geq 3$.
\item The generalized Petersen graph $P(n,2)$. In Section \ref{sec-petersen} we show that $\Gamma(P(n,2)) = n$ for all $n \geq 3$.\end{enumerate}

The rationale behind choosing these graph families for the experiments is as follows. Graphs based on chess boards and pieces have been widely studied for domination problems including upper domination (e.g. see \cite{hedetniemi}), and provide non-trivial and structurally interesting instances. Grid graphs have also been extensively considered for domination problems \cite{alanko,dorfling,goncalves,gravier}; the result showing that the domination number of grid graphs separates into 23 individual cases is a wonderful example of the surprising complexity of these instances.

We are also interested in exploring how graph density impacts on the performance of the formulations. The size of Formulation 1 (both before and after augmenting the additional constraints) depends only on the number of vertices, while the size of (Augmented) Formulation 2 depends on the number of vertices and edges. As such, it is reasonable to expect that the relative performance of these formulations might be impacted by the graph density. To study this, we have selected graph families with a variety of average degrees. For the queen, rook and bishop graphs, the average degree increases with the size of the graph. For the knight and grid graphs, the average degree tends towards 8 and 4 respectively as the size increases. We have also included flower snarks and generalized Petersen graphs as widely studied examples of 3--regular graphs.

Finally, we will also consider two kinds of random graphs; Erd\H{o}s-R\'{e}nyi graphs, and unit disk graphs. In particular, the latter have been widely considered in the context of domination problems \cite{clark,marathe,shang,wang}. For both Erd\H{o}s-R\'{e}nyi graphs and unit disk graphs, the parameters can be set to ensure that the expected average degree is equal to a desired value, which enables us to further explore the impact of graph density on the performance of the formulations. In the upcoming experiments, we consider average degrees of 4, 6, 8 and 10.

The time taken to solve the instances for various choices of parameters are displayed in Figures \ref{fig-udgtest}--\ref{fig-ertest}. In those figures, the large dotted line corresponds to Formulation 1 in this paper, the small dotted line to Formulation 2, the solid line to Augmented Formulation 1, and the dashed line to Augmented Formulation 2.

\input{ud_plots}

We first comment on the performance of the augmented formulations compared to the non-augmented formulations. As discussed in Section \ref{sec-additional}, the augmented formulations have a reduced set of feasible solutions, however this does not necessarily correspond to a faster computation time for a solver. For instance, there may be additional overhead costs associated with handling the additional constraints, and there may also be fewer optimal solutions for the solver to ``stumble into". However, we observe that in our experiments, the augmented formulations did indeed outperform their equivalent non-augmented formulations, particularly as the size of the graph increased. In particular, Augmented Formulation 1 performed significantly better than Formulation 1 in all experiments. Augmented Formulation 2 generally outperformed Formulation 2, but the level of improvement was not as stark as for the first formulation, and there were some graph families where Formulation 2 displayed (marginally) superior performance even as the size of the graph increased.

For graphs where the average degree increases with the size of the graph (e.g. the queens, rook, and bishop graphs), we observed that Augmented Formulation 1 tended to outperform Augmented Formulation 2. However, for 3--regular graphs (e.g. the Flower snarks and the generalized Petersen graphs) Augmented Formulation 2 significantly outperformed Augmented Formulation 1. We postulate that there is some tipping point for the average degree where the two formulations perform equally well, and indeed, we observe this for the randomly generated graphs with low average degree (e.g. Erd\H{o}s-R\'{e}nyi and unit disk graphs). For the randomly generated graphs with average degree 4, Augmented Formulation 2 outperforms Augmented Formulation 1. However, the gap closes as the average degree is increased, with the two formulations performing similarly well for Erd\H{o}s-R\'{e}nyi graphs of average degree 6, and unit disk graphs of average degree 10.

In conclusion, these experiments suggest that, among the formulations introduced in this paper, Augmented Formulation 1 should be favoured for all but very sparse instances, in which case Augmented Formulation 2 should be favoured instead. A fascinating topic for future research would be to consider if either of these formulations could be further tightened with additional constraints, or by considering different variables. Alternatively, it would be interesting to investigate whether the constraints developed in this paper could be employed directly to construct more sophisticated exact algorithms for upper domination, rather than simply embedding them in an MILP formulation.

\section{Upper domination numbers for generalized Petersen graphs\label{sec-petersen}}

We conclude this paper by deriving the upper domination numbers for all generalized Petersen graphs, a family of graphs first introduced by Coxeter \cite{coxeter} and later named by Watkins \cite{watkins}. We first recall the definition of generalized Petersen graphs.

\begin{definition}Consider integers $n \geq 3$, and $1 \leq k < n/2$. Then the generalized Petersen graph with parameters $n,k$ (denoted $P(n,k)$) is a 3--regular graph on $2n$ vertices. The vertex set is $\{u_0$, $u_1$, $\hdots$, $u_{n-1}$, $v_0$, $v_1$, $\hdots$, $v_{n-1}\}$, and the edge set is constructed as follows. For each $0 \leq i \leq n-1$, $P(n,k)$ contains edges $\{u_i,u_{i+1}\}$, $\{u_i,v_i\}$, and $\{v_i,v_{i+k}\}$, where the subscripts are to be read modulo $n$.\end{definition}

In the upcoming proof, we will make use of the following two results from \cite{bazgan} and \cite{rosenfeld}, where $|V(G)|$ is the order of the graph $G$, and $\alpha(G)$ is the independence number of $G$.

\begin{lemma}[Bazgan et al. \cite{bazgan}]For any given graph $G$ of minimum degree $\delta$ and maximum degree $\Delta$, we have:

$$\alpha(G) \leq \Gamma(G) \leq \max \left\{\alpha(G), \frac{|V(G)|}{2} + \frac{\alpha(G)(\Delta - \delta)}{2\Delta} - \frac{\Delta - \delta}{\Delta}\right\}.$$\label{lem-bazgan}\end{lemma}

\begin{lemma}[Rosenfeld \cite{rosenfeld}]If $G$ is a $d$--regular graph with $d \leq \left\lceil\frac{|V(G)|}{2}\right\rceil$, then $\alpha(G) \leq \left\lfloor\frac{|V(G)|}{2}\right\rfloor$.\end{lemma}

These two results lead immediately to the following corollary.

\begin{corollary}If $G$ is a 3--regular graph with $|V(G)| \geq 6$, then $\Gamma(G) \leq \frac{|V(G)|}{2}$.\label{cor-gp}\end{corollary}

\begin{theorem}Consider the generalized Petersen graph $P(n,k)$, for $n \geq 3$ and $1 \leq k < n/2$. Then, $\Gamma(P(n,k)) = n$.\end{theorem}

\begin{proof}By definition, $P(n,k)$ is a 3--regular graph on $2n \geq 6$ vertices. From Corollary \ref{cor-gp}, we have $\Gamma(P(n,k)) \leq n$. Then, it suffices to provide a corresponding lower bound. From the definition of $P(n,k)$ we can see that vertices $u_i$ and $v_j$ are adjacent if and only if $i = j$. Then it is clear that $S = \{u_0, \hdots, u_{n-1}\}$ is a dominating set, and that any proper subset of $S$ will leave at least one vertex $v_i$ undominated. Therefore $S$ is a minimal dominating set, and so $\Gamma(P(n,k)) \geq n$, completing the proof.\end{proof}

\section*{Declarations}

The authors have no relevant financial or non-financial interests to disclose.

\bibliographystyle{plain}

\end{document}